\documentclass[12pt]{amsproc}

\newtheorem{theorem}{\sc Theorem}[section]
\newtheorem{lemma}[theorem]{\sc Lemma}
\newtheorem{proposition}[theorem]{\sc Proposition}
\newtheorem{corollary}[theorem]{\sc Corollary}

\begin{document}

\author{Pavel Shumyatsky}
\address{Department of Mathematics, University of Brasilia, 70910-900 Bras\'ilia DF, Brazil}
\email{pavel@unb.br}

\author{Danilo San\c c\~ao da Silveira}
\address{Department of Mathematics, University of Brasilia, 70910-900 Bras\'ilia DF, Brazil}
\email{sancaodanilo@ufg.br}


\thanks{This work was supported by CNPq and FAPDF, Brazil. }

\dedicatory{For Evgeny Khukhro on his 60th birthday}
\title[On groups with automorphisms]{On finite groups with automorphisms\\ whose fixed points are Engel}

\begin{abstract} The main result of the paper is the following theorem. Let $q$ be a prime, $n$ a positive integer and $A$ an elementary abelian group of order $q^2$. Suppose that $A$ acts coprimely on a finite group $G$ and assume that for each $a\in A^{\#}$ every element of $C_{G}(a)$ is $n$-Engel in $G$. Then the group $G$ is $k$-Engel for some $\{n,q\}$-bounded number $k$.
\end{abstract}

\maketitle

\section{Introduction}
Let $A$ be a  finite group acting on a finite group $G$. Many well-known results show that the structure of the centralizer $C_G(A)$ (the fixed-point subgroup) of $A$ has influence over the structure of $G$. The influence is especially strong if $(|A|,|G|)=1$, that is, the action of $A$ on $G$ is coprime. We will give examples illustrating this phenomenon.

Thompson proved that if $A$ is of prime order and $C_G(A)=1$, then $G$ is nilpotent \cite{T}. Higman proved that any nilpotent
group admitting a fixed-point-free automorphism of prime order $q$ has nilpotency class bounded by some function $h(q)$ depending on $q$ alone \cite{Higman}. The reader can find in \cite{Khu1} and \cite{Khu2} an account on the more recent developments related to these results. The next result is a consequence of the classification of finite simple groups \cite{Wang}: If $A$ is a group of automorphisms of $G$ whose order is coprime to that of $G$ and $C_G(A)$ is nilpotent or has odd order, then $G$ is soluble. Once the group $G$ is known to be soluble, there is a wealth of results bounding the Fitting height of $G$ in terms of the order of $A$ and the Fitting height of $C_G(A)$. This direction of research was started by Thompson in \cite{thompson2}. The proofs mostly use representation theory in the spirit of the Hall-Higman work \cite{HH}. A general discussion of these methods and their use in numerous fixed-point theorems can be found in Turull \cite{Tu}.

Following the solution of the restricted Burnside problem it was discovered that the exponent of $C_G(A)$ may have strong impact over the exponent of $G$. Remind that a group $G$ is said to have exponent $n$ if $x^n=1$ for every $x\in G$ and $n$ is the minimal positive integer with this property. The next theorem was obtained in \cite{KS}.

\begin{theorem}\label{q2}
Let $q$ be a prime, $n$ a positive integer and $A$ an elementary abelian group of order $q^2$. Suppose that $A$ acts coprimely on a finite group $G$ and assume that $C_{G}(a)$ has exponent dividing $n$ for each $a\in A^{\#}$. Then the exponent of $G$ is $\{n,q\}$-bounded.
\end{theorem}  
Here and throughout the paper $A^{\#}$ denotes the set of nontrivial elements of $A$. The proof of  the above result involves a number of deep ideas. In
particular, Zelmanov's techniques that led to the solution of the restricted Burnside problem \cite{Z0,Z1} are combined with the Lubotzky--Mann theory of powerful $p$-groups \cite{luma}, Lazard's criterion for a pro-$p$ group to be $p$-adic analytic \cite{L}, and a theorem of Bakhturin and Zaicev on Lie algebras admitting a group of automorphisms whose fixed-point subalgebra is PI \cite{BZ}.

Let $n$ be a positive integer and let $x,y$ be elements of a group $G$. The commutators $[x,_n y]$ are defined inductively by the rule
$$[x,_0 y]=x,\quad [x,_n y]=[[x,_{n-1} y],y].$$
An element $x$ is called a (left) $n$-Engel element if for any $g\in G$ we have $[g,_n x]=1$. A group $G$ is called $n$-Engel if all elements of $G$ are $n$-Engel. The main result of the present article is the following theorem.

\begin{theorem}\label{main} Let $q$ be a prime, $n$ a positive integer and $A$ an elementary abelian group of order $q^2$. Suppose that $A$ acts coprimely on a finite group $G$ and assume that for each $a\in A^{\#}$ every element of $C_{G}(a)$ is $n$-Engel in $G$. Then the group $G$ is $k$-Engel for some $\{n,q\}$-bounded number $k$.
\end{theorem}

The methods employed in the proof of Theorem \ref{main} are roughly the same as those in the proof of Theorem \ref{q2}. There are however some rather significant differences. The main difficulty is that Theorem \ref{main} does not reduce as quickly as Theorem \ref{q2} to questions about Lie algebras over field. Instead, we have to work with Lie rings. This hurdle is overcome via Proposition \ref{ZelBazaRing} obtained in the next section and some other tools. 

Throughout the paper we use without special references the well-known properties of coprime actions:  

If $\alpha$ is an automorphism of a finite group $G$ of coprime order, $(|\alpha|,|G|)=1$, then $ C_{G/N}(\alpha)=C_G(\alpha)N/N$ for any $\alpha$-invariant normal subgroup $N$. 

If $A$ is a noncyclic abelian group acting coprimely on a finite group $G$, then $G$ is generated by the subgroups $C_G(B)$, where $A/B$ is cyclic.

We use the expression ``$\{a,b,\dots\}$-bounded'' to abbreviate ``bounded from above in terms of  $a, b,\dots$ only''.

\section{About Lie Rings and Lie Algebras}
%%%%%%%%%%%%%%%%%%%%%%%%%%%%%%%%%%%%%%%%%%%%%%%%%%%%%%%%%%%%%%%%%%%%%%%%%%%%%%%%%
%%%%%%%%%%%%%%%%%%%%%%%%%%%%%%%%%%%%%%%%%%%%%%%%%%%%%%%%%%%%%%%%%%%%%%%%%%%%%%%%%

Let $X$ be a subset of a Lie algebra $L$. By a commutator in elements of $X$ we mean any element of $L$ that can be obtained as a Lie product of elements of $X$ with some system of brackets. If $x_1,\ldots,x_k,x, y$ are elements of $L$, we define inductively 
$$[x_1]=x_1; [x_1,\ldots,x_k]=[[x_1,\ldots,x_{k-1}],x_k]$$
and 
$[x,_0y]=x; [x,_my]=[[x,_{m-1}y],y],$ for all positve integers $k,m$.  
As usual, we say that an element $a\in L$ is ad-nilpotent if there exists a positive integer $n$ such that $[x,_na]=0$ for all $x\in L$. If $n$ is the least integer with the above property then we say that $a$ is ad-nilpotent of index $n$. 

The next theorem is a deep result of Zelmanov \cite{Z1}.

\begin{theorem}\label{Z1992}
Let $L$ be a Lie algebra over a field and suppose that $L$ satisfies a polynomial identity. If $L$ can be generated by a finite set $X$ such that every commutator in elements of $X$ is ad-nilpotent, then $L$ is nilpotent.
\end{theorem}

An important criterion for a Lie algebra to satisfy a polynomial identity is provided by the next theorem. It was proved by Bakhturin and Zaicev for soluble groups $A$ \cite{BZ} and later extended by Linchenko to the general case \cite{l}. 
 
\begin{theorem}\label{blz}
Let $L$ be a Lie algebra over a field $K$. Assume that a finite group $A$ acts on $L$ by automorphisms in such a manner that $C_L(A)$ satisfies a polynomial identity. Assume further that the characteristic of $K$ is either $0$ or prime to the order of $A$. Then $L$ satisfies a polynomial identity.
\end{theorem}

Both Theorem \ref{Z1992} and Theorem \ref{blz} admit respective quantitative versions (see for example \cite{aaaa}).

\begin{theorem}\label{Z1}
Let $L$ be a Lie algebra over a field $K$ generated by $a_1,\ldots,a_m$. Suppose that $L$ satisfies a polynomial identity $f\equiv 0$ and each commutator in $a_1,\ldots,a_m$ is ad-nilpotent of index at most $n$. Then $L$ is nilpotent of $\{f,K,m,n\}$-bounded class.
\end{theorem}

\begin{theorem}\label{bZ1}
 Let $L$ be as in Theorem \ref{blz} and assume that $C_L(A)$ satisfies a polynomial identity $f\equiv 0$. Then $L$ satisfies a polynomial identity of $\{|A|,f,K\}$-bounded degree.
\end{theorem}

By combining the above results we obtain the following corollary. 

\begin{corollary}\label{ZelBazaField}
Let $L$ be a Lie algebra over a field $K$ and $A$ a finite group of automorphisms of $L$ such that $C_L(A)$ satisfies the polynomial identity $f\equiv 0$. Suppose that the characterisitic of $K$ is either $0$ or prime to the order of $A$. Assume that $L$ is generated by an $A$-invariant set of $m$ elements in which every commutator is ad-nilpotent of index at most $n$. Then $L$ is nilpotent of $\{|A|,f,K,m,n\}$-bounded class.
\end{corollary}

We will need a similar result for Lie rings. As usual, $\gamma_i(L)$ denotes the $i$th term of the lower central series of $L$.
\begin{proposition}\label{ZelBazaRing}
Let $L$ be a Lie ring and $A$ a finite group of automorphisms of $L$ such that $C_L(A)$ satisfies the polynomial identity $f\equiv 0$. Further, assume that $L$ is generated by an $A$-invariant set of $m$ elements such that every commutator in the generators is ad-nilpotent of index at most $n$. Then there exist positive integers $e$ and $c$, depending only on $|A|, f, m$ and $n$, such that  $e\gamma_c(L)=0$.
\end{proposition}

\begin{proof} Since the set of generators of $L$ is $A$-invariant, every automorphism $\alpha\in A$ induces a permutation of the generators.
Let $L_{\mathbb Q}$ denote the free Lie algebra on free generators $x_1,\dots,x_m$ over the field ${\mathbb Q}$ of rational numbers and  $L_{\mathbb Z}$ the Lie subring of $L_{\mathbb Q}$ generated by $x_1,\dots,x_m$. For each $\alpha\in A$ let $\phi_{\alpha}$ be the automorphism of $L_{\mathbb Q}$ (or of $L_{\mathbb Z}$) such that $x_i^{\phi_\alpha}=x_{\pi(i)}$ where $\pi$ is the permutation that $\alpha$ induces on the generators of $L$. The mapping that takes $\alpha\in A$ to $\phi_{\alpha}$ induces a natural action of $A$ on $L_{\mathbb Q}$ (or on $L_{\mathbb Z}$) by automorphisms.

Let $I$ be the ideal of $L_{\mathbb Z}$ generated by all values of $f$ on elements of $C_{L_{\mathbb Z}}(A)$ and by all elements of the form $[u,{}_nv]$, where $u$ ranges through $L_{\mathbb Z}$ and $v$ ranges through the set of all commutators in the generators  $x_1,\dots,x_m$. By Corollary \ref{ZelBazaField} $L_{\mathbb Q}/{\mathbb Q}I$ is nilpotent of $\{|A|,f,m,n\}$-bounded class, say $c-1$.

If $y_1,\dots,y_{c}$ are not necessarily distinct elements of the set $x_1,\dots,x_m$, it follows that the commutator $[y_1,\dots,y_{c}]$ lies in ${\mathbb Q}I$ and therefore it can be written in the form $$[y_1,\dots,y_{c}]=\sum_iq_if_iw_i,$$ where $q_i$ are rational numbers, $f_i$ are either elements of the form $[u,{}_nv]$ as above or values of $f$ in elements of $C_{L_{\mathbb Z}}(A)$ and $f_iw_i$ are elements of $L_{\mathbb Z}$ obtained by multiplying (several times) $f_i$ with products in the free generators $x_1,\dots,x_m$. 

Let $e$ be the least common multiple of the denominators of the coefficients $q_i$ taken over all possible choices of $y_1,\dots,y_{c}$ in $\{x_1,\dots,x_m\}$. Then $e[y_1,\dots,y_{c}]\in I$ for any choice of $y_1,\dots,y_c$. Therefore $e\gamma_c(L_{\mathbb Z})\leq I$.

We remark that there exists a natural homomorphism of $L_{\mathbb Z}/I$ onto $L$ under which each automorphism $\phi_{\alpha}$ induces the automorphism $\alpha$ of $L$. Thus, $e\gamma_c(L)=0$, as required.
\end{proof}
We will now quote a useful lemma from \cite{KS}.

\begin{lemma}\label{lemmapskh}
Let $L$ be a Lie ring and $H$ a subring of $L$ generated by $m$ elements $h_1,\ldots,h_m$ such that all commutators in $h_i$ are ad-nilpotent in $L$ of index at most $n$. If $H$ is nilpotent of class $c$, then for some $\{c,m,n\}$-bounded number $u$ we have $[L,\underbrace{H,\ldots,H}_u]=0$.
\end{lemma}

Recall that the identity $$\sum_{\sigma\in S_n}[y,x_{\sigma(1)},\ldots, x_{\sigma(n)}]\equiv 0$$ is called linearized $n$-Engel identity. In general, Theorem \ref{Z1992} cannot be extended to the case where $L$ is just a Lie ring (rather than Lie algebra over a field). However such an extension does hold in the particular case where the polynomial identity $f\equiv 0$ is a linearized Engel identity.

\begin{theorem}\label{Znilpotenteanel} Let $f$ be a Lie polynomial of degree $n$ all of whose coefficients equal 1 or -1. Let $L$ be a Lie ring generated by finitely many elements $a_1,\ldots,a_m$ such that all commutators in the generators are ad-nilpotent of index at most $n$. Assume that $L$ satisfies the identity $f\equiv 0$. Then $L$ is nilpotent of $\{f,m,n\}$-bounded class.
\end{theorem}

The deduction of Theorem \ref{Znilpotenteanel} from Theorem \ref{Z1} uses just standard arguments. For the reader's convenience we sketch out a proof.
\begin{proof} By Proposition \ref{ZelBazaRing} applied with $A=1$, there exist positive integers $e$ and $c$, depending only on $f,m$ and $n$, such that  $e\gamma_c(L)=0$. We set $M=eL$ and notice that $M$ is a nilpotent ideal of $L$ whose nilpotency class is bounded in terms of $f,m$ and $n$. Let $d$ be the derived length of $M$ and use induction on $d$. Suppose first that $M=0$. Then $L$ is a direct sum of its primary components $L_p$, where $p$ ranges through prime divisors of $e$. It is sufficient to show that the nilpotency class of each $L_p$ is bounded in terms of $f, m$ and $n$ and so without loss of generality we can assume that $e=p^k$ for a prime $p$. We note that $k$ here is $\{f,m,n\}$-bounded. The quotient $L/pL$ can be regarded as a Lie algebra over the field with $p$ elements and so it is nilpotent of bounded class by Theorem \ref{Z1}. Hence, there exists an $\{f,m,n\}$-bounded number $c_1$ such that $\gamma_{c_1}(L)\leq pL$. This implies that $\gamma_{kc_1}(L)\leq p^kL=0$ and we are done.

Assume now that $M\neq0$ and let $D$ be the last nontrivial term of the derived series of $M$. By induction $\bar{L}=L/D$ is nilpotent of bounded class. The ring $\bar{L}$ naturally acts on $D$ by derivations and Lemma \ref{lemmapskh} shows that for some $\{f,m,n\}$-bounded number $u$ we have $[D,\underbrace{\bar{L},\ldots,\bar{L}}_u]=0$. It follows that $L$ is nilpotent of $\{f,m,n\}$-bounded class, as required.
\end{proof}
\section{On associated Lie rings} 

There are several well-known ways to associate a Lie ring to a group (see \cite{Hu2,Khu1,aaaa}). For the reader's convenience we will briefly describe the construction that we are using in the present paper.

Let $G$ be a group. A series of subgroups $$G=G_1\geq G_2\geq\dots\eqno{(*)}$$ is called an $N$-series if it satisfies $[G_i,G_j]\leq G_{i+j}$ for all $i,j$. Obviously any $N$-series is central, i.e. $G_i/G_{i+1}\leq Z(G/G_{i+1})$ for any $i$. Given an $N$-series $(*)$, let $L^*(G)$ be the direct sum of the abelian groups $L_i^*=G_i/G_{i+1}$, written additively. Commutation in $G$ induces a binary operation $[,]$ in $L^*(G)$. For homogeneous elements $xG_{i+1}\in L_i^*,yG_{j+1}\in L_j^*$ the operation is defined by $$[xG_{i+1},yG_{j+1}]=[x,y]G_{i+j+1}\in L_{i+j}^*$$ and extended to arbitrary elements of $L^*(G)$ by linearity. It is easy to check that the operation is well-defined and that $L^*(G)$ with the operations $+$ and $[,]$ is a Lie ring. If all quotients $G_i/G_{i+1}$ of an $N$-series $(*)$ have prime exponent $p$ then $L^*(G)$ can be viewed as a Lie algebra over the field with $p$ elements. Any automorphism of $G$ in the natural way induces an automorphism of $L^*(G)$. If $G$ is finite and $\alpha$ is an automorphism of $G$ such that $(|\alpha|,|G|)=1$, then the subring of fixed points of $\alpha$ in $L^*(G)$ is isomorphic with the Lie ring associated to the group $C_G(\alpha)$ via the series formed by intersections of $C_G(\alpha)$ with the series $(*)$.

 In the case where the series $(*)$ is just the lower central series of $G$ we write $L(G)$ for the associated Lie ring. In the case where the series $(*)$ is the $p$-dimension central series of $G$ we write $L_p(G)$ for the subalgebra generated by the first homogeneous component $G_1/G_2$ in the associated Lie algebra over the field with $p$ elements.

\begin{proposition}\label{pppp}
Let $G$ be an $m$-generated group satisfying the hypothesis of Theorem \ref{main}. Let $p$ be a prime number that divides the order of $G$. Then
\begin{enumerate}
\item $L_p(G)$ is nilpotent of $\{p,q,m,n\}$-bounded class.
\item There exists positive integers $e,c$ depending only on $m, n$ and $q$, such that $e\gamma_c(L(G))=0$.
\end{enumerate}
\end{proposition}
The proofs of the statements (1) and (2) are similar. We will give a detailed proof of the second statement only. The proof of (1) can be obtained simply by replacing every appeal to Proposition \ref{ZelBazaRing} in the proof of (2) by an appeal to Corollary \ref{ZelBazaField}. An important observation that is used in the proof is that if $G$ is a finite $p$-group that can be generated by $d$ elements, then we can choose $d$ generators from any subset that generates $G$. This follows from the well-known Burnside Basis Theorem \cite{Huppert}.
\begin{proof} Since $G$ is generated by the centralizers $C_G(a)$, where $a\in A^\#$, and since all elements in the centralizers $C_G(a)$ are Engel, the Baer Theorem  [\cite{Huppert}, III,6.14] implies that the group $G$ is nilpotent. Therefore $G$ is a direct product of its Sylow subgroups. Hence, without loss of generality we can assume that $G$ is a $p$-group.
Let $\gamma_j=\gamma_j(G)$, $L_j=\gamma_j/\gamma_{j+1}$ and $L=L(G)=\oplus_{j\geq 1}L_j$ be the Lie ring associated with the group $G$. Let $A_1,\ldots,A_{q+1}$ be the distinct maximal subgroups of $A$. Set $L_{ij}=C_{L_j}(A_i)$. We know that any $A$-invariant subgroup is generated by the centralizers of $A_i$. Therefore for any $j$ we have 
$$L_j=\sum_{i=1}^{q+1}L_{ij}.$$
Further, for any $l\in L_{ij}$ there exists $x\in \gamma_{j}\cap C_G(A_i)$ such that $l=x\gamma_{j+1}$. Since  $x$ is $n$-Engel in $G$, it follows that $l$ is ad-nilpotent of index at most $n$. Thus,

\begin{equation}\label{adnilpotent1}
\mbox{any element in}\  L_{ij}\ \mbox{is ad-nilpotent in $L$ of index at most}\ n.
\end{equation} 
Since $G$ is generated by $m$ elements, the additive group $L_1$ is generated by $m$ elements. It follows that the Lie ring $L$ is generated by at most $m$ ad-nilpotent elements, each from $L_{i1}$ for some $i$. 

Let $\omega$ be a primitive $q$th root of unity and consider the tensor product $\overline{L}=L\otimes\mathbb{Z}[\omega]$. Set $\overline{L}_j=L_j\otimes\mathbb{Z}[\omega]$ for $j=1,2,\dots$. We regard $\overline{L}$ as a Lie ring and so $\overline{L}=\left<\overline{L}_1\right>$. Since the additive subgroup $L_1$ is generated by $m$ elements, it follows that the additive subgroup $\overline{L}_1$ is generated by $(q-1)m$ elements. We also remark that there is a natural embedding of the ring $L$ into the ring $\overline{L}$.

The group $A$ acts on $\overline{L}$ in the natural way and we have $\overline{L}_{ij}=C_{\overline{L}_j}(A_i)$ where $\overline{L}_{ij}=L_{ij}\otimes\mathbb{Z}[\omega]$. We will now establish the following claim.

\begin{equation}\label{Lijbarra}
\mbox{Any element in}\  \overline{L}_{ij}\ \mbox{is ad-nilpotent in $\overline{L}$ with $\{n,q\}$-bounded index.}
\end{equation} 

Indeed choose $y\in \overline{L}_{ij}$ and write $$y=x_0+\omega x_1+\omega^2x_2+\cdots+\omega^{q-2}x_{q-2}$$ for suitable $x_s\in L_{ij}$. In view of (\ref{adnilpotent1}) it is easy to see that each of the summands $\omega^sx_s$ is ad-nilpotent in $\overline{L}$ of index at most $n$. 

Let $H=\left<x_0,\omega x_1, \omega^2x_2,\ldots, \omega^{q-2}x_{q-2}\right>$ be the subring of $\overline{L}$ generated by $x_0,\omega x_1,\omega^2x_2,\ldots,\omega^{q-2}x_{q-2}$. We wish to show that $H$ is nilpotent of $\{n,q\}$-bounded class. 

Note that $H\leq C_{\overline{L}}(A_i)$ since $\omega^sx_s\in C_{\overline{L}}(A_i)$ for all $s$. A commutator of weight $t$ in the elements $x_0,\omega x_1,\omega^2x_2,\ldots,\omega^{q-2}x_{q-2}$ has the form $\omega^tx$, for some $x\in L_{ij_0},$ where $j_0=tj$. By (\ref{adnilpotent1}), the element $x$ is ad-nilpotent of index at most $n$ in $L$ and so we deduce that $\omega^tx$ is ad-nilpotent of index at most $n$ in $\overline{L}$.

Since $C_G(A_i)$ is $n$-Engel, we conclude that $C_L(A_i)$ satisfies the linearized $n$-Engel identity. This identity is multilinear and so it is also satisfied by $C_L(A_i)\otimes \mathbb{Z}[\omega]=C_{\overline{L}}(A_i)$. Since $H\leq C_{\overline{L}}(A_i)$, it follows that the identity is satisfied by $H$. Hence, by Theorem \ref{Znilpotenteanel}, $H$ is nilpotent of $\{n,q\}$-bounded class.

By Lemma \ref{lemmapskh}, there exists a positive integer $v$, depending only on $n$ and $q$, such that $[\overline{L},_v H]=0$. Since $y\in H$, we conclude that $y$ is ad-nilpotent in $\overline{L}$ with $\{n,q\}$-bounded index. This proves Claim (2). 

An element $x\in\overline{L}$ will be called a common ``eigenvector'' for $A$ if for any $a\in A$ there exists $s$ such that $x^a=\omega^sx$. Since $(|A|,|G|)=1$, the additive group of the Lie ring $\overline{L}$ is generated by common eigenvectors for $A$ (see for example \cite[Lemma 4.1.1]{Khu1}). 

We have already remarked that the additive subgroup $\overline{L}_1$ is generated by $(q-1)m$ elements. It follows that the Lie ring $\overline{L}$ is generated by at most $(q-1)m$ common eigenvectors for $A$. Certainly, the Lie ring $\overline{L}$ is generated by an $A$-invariant set of at most $(q-1)qm$ common eigenvectors for $A$.

Since $A$ is noncyclic, any common eingenvector for $A$ is contained in the centralizer $C_{\overline{L}}(A_i)$ for some $i$. Further, any commutator in common eigenvectors is again a common eigenvector. Therefore if $l_1,\ldots,l_{s}\in \overline{L}_1$ are common eigenvectors for $A$ generating $\overline{L}$, then any commutator in these generators belongs to some $\overline{L}_{ij}$ and therefore, by Claim (2), is ad-nilpotent of $\{n,q\}$-bounded index.

As we have mentioned earlier, for any $i$ the subalgebra $C_{\overline{L}}(A_i)$ satisfies the linearized $n$-Engel identity. 
Thus, by Proposition \ref{ZelBazaRing}, there exist positive integers $e, c$ depending only on $m,n$ and $q$ such that $e\gamma_c(\overline{L})=0$. Since $L$ embeds into $\overline{L}$, we also have $e\gamma_c(L)=0$. The proof is complete. 
\end{proof}

\section{Proof of the main theorem}

A well-known theorem of Gruenberg says that a soluble group generated by finitely many Engel elements is nilpotent (see \cite[12.3.3]{Rob}). We will require a quantitative version of this theorem.
\begin{lemma}\label{gru}
Let $G$ be a group generated by $m$ $n$-Engel elements and suppose that $G$ is soluble with derived length $d$. Then $G$ is nilpotent of $\{d,m,n\}$-bounded class. 
\end{lemma}
\begin{proof} Suppose that the lemma is false. Then for each positive integer  $i$ we can choose a group $G_i$ satisfying the hypothesis of the lemma and having nilpotency class at least $i$. In each group $G_i$ we fix generators $g_{i1},g_{i2},\ldots, g_{im}$ that are $n$-Engel elements.

In the Cartesian product of the groups $G_i$ consider the subgroup $D$ generated by $m$ elements
$$g_1=(g_{11},g_{21},\ldots),\\
\ldots,\\
g_m=(g_{1m},g_{2m}\ldots).$$ 
Thus, $D$ is a soluble group generated by $m$ elements that are $n$-Engel. By Gruenberg's theorem $D$ is nilpotent, say of class $c$. We remark that each of the groups $G_i$ is isomorphic to a quotient $D$. Hence each of the groups $G_i$ is nilpotent of class at most $c$. A contradiction.
\end{proof}

Given a group $H$ and an integer $i$, in what follows we write $H^i$ for the subgroup of $H$ generated by all $i$th powers and $\gamma_i(H)$ for the $i$th term of the lower central series of $H$. 
We will also require the concept of powerful $p$-groups. These were introduced by Lubotzky and Mann in \cite{luma}: a finite $p$-group $G$ is powerful if and only if $G^p\geq [G,G]$ for $p\ne 2$ (or $G^4\geq [G,G]$ for $p=2$). The reader can consult books \cite{GA} or \cite{Khu2} for more information on these groups.

Powerful $p$-groups have many nice properties. In particular, if $P$ is a powerful $p$-group, the subgroups $\gamma_i(P)$ and $P^i$ are also powerful. Moreover, if $n_1,\dots,n_s$ are positive integers, it follows by repeated applications of \cite[Props. 1.6 and 4.1.6]{luma} that $$[P^{n_1},\dots,P^{n_s}]\leq\gamma_s(P)^{n_1\cdots n_s}.$$ If $P$ is generated by $d$ elements, then any subgroup of $P$ can be generated by at most $d$ elements and $P$ is a product of $d$ cyclic subgroups (see Theorem 1.11 and Theorem 1.12 in \cite{luma}).

We are now ready to prove Theorem \ref{main}.
\begin{proof}[Proof of Theorem \ref{main}]
Let $G$ be as in Theorem \ref{main}. We wish to show that $G$ is $k$-Engel for some $\{n,q\}$-bounded number $k$. Since $G$ is generated by the centralizers $C_G(a)$ with $a\in A^\#$ and since all elements in the centralizers $C_G(a)$ are Engel, the Baer Theorem implies that the group $G$ is nilpotent. Choose arbitrarily $x,y\in G$. It is sufficient to show that the subgroup $\langle x,y\rangle$ is nilpotent of $\{n,q\}$-bounded class. Without loss of generality we can assume that no proper $A$-invariant subgroup of $G$ contains both $x$ and $y$. In other words, we will assume that $G=\langle x^A,y^A\rangle$. Therefore the group $G$ can be generated by 2$q^2$ elements. Every Sylow subgroup of $G$ satisfies these hypotheses and therefore without loss of generality we can assume that $G$ is a $p$-group for some prime $p\neq q$. Using the fact that $G$ is generated by the centralizers $C_G(a)$ with $a\in A^\#$ and combining it with the Burnside Basis Theorem we can assume further that $G$ is generated by at most 2$q^2$ elements which are $n$-Engel. In view of Lemma \ref{gru} it is now sufficient to show that the derived length of $G$ is $\{n,q\}$-bounded. Let $L=L(G)$. By Proposition \ref{pppp} (2), there exist positive integers $e, c$ that depend only on $n$ and $q$ such that $e\gamma_c(L)=0$. If $p$ is not a divisor of $e$, we have $\gamma_c(L)=0$ and so the group $G$ is nilpotent of class at most $c-1$. In that case the proof is complete and so we assume that $p$ is a divisor of $e$. Consider first the particular case where $G$ is a powerful $p$-group. Set $R=G^e$ and assume that $R\neq1$.

We have $$[R,R]\leq[G,G]^{e^2}\leq G^{pe^2}=R^{pe}.$$ Let $L_1=L(R)$. By Proposition \ref{pppp} (2), $e\gamma_c(L_1)=0$. The fact that $e\gamma_c(L_1)=0$ means that $\gamma_c(R)^e\leq\gamma_{c+1}(R)$. Taking into account that $R$ is powerful, we write $$\gamma_c(R)^e\leq\gamma_{c+1}(R)=[R',{}_{c-1}R]\leq[R^{pe},{}_{c-1}R]\leq\gamma_c(R)^{pe}.$$ Hence, $\gamma_c(R)^e=1$. Since $\gamma_c(R)$ is powerful and generated by at most 2$q^2$ elements, we conclude that $\gamma_c(R)$ is a product of at most 2$q^2$ cyclic subgroups. Hence the order of $\gamma_c(R)$ is at most $e^{2q^2}$. It follows that the derived length of $R$ is $\{n,q\}$-bounded. Recall that $G$ is a powerful $p$-group and $R=G^e$. It follows that the derived length of $G$ is $\{n,q\}$-bounded. Thus, we proved the desired result in the case where $G$ is a powerful $p$-group. 

Let us now drop the assumption that $G$ is powerful. By Proposition \ref{pppp} (1) the algebra $L_p(G)$ is nilpotent with bounded nilpotency class. Proposition 1 of \cite{KS} now tells us that $G$ has a characteristic powerful subgroup of bounded index. We already know that the derived length of the powerful subgroup is $\{n,q\}$-bounded. Hence, the derived length of $G$ is $\{n,q\}$-bounded as well. The proof is now complete.
\end{proof}

\end{document}